\documentclass{amsart}
\usepackage{amssymb}
\usepackage{tikz}
\usetikzlibrary{matrix,arrows}
\theoremstyle{plain}
\newtheorem{thm}{Theorem}[section]
\newtheorem{lem}[thm]{Lemma}

\newtheorem{prop}[thm]{Proposition}

\newtheorem{defn}[thm]{Definition}

\newtheorem{example}[thm]{Example}

\newtheorem{rem}[thm]{Remark}

\newcommand{\Hom}{\mathop{\mathrm{Hom}}\nolimits}

\DeclareMathOperator{\GL}{GL}
\DeclareMathOperator{\SO}{SO}
\DeclareMathOperator{\SL}{SL}
\DeclareMathOperator{\PSL}{PSL}
\DeclareMathOperator{\cd}{cd}
\DeclareMathOperator{\tr}{tr}
\newcommand{\Z}{\mathbb Z}
\newcommand{\R}{\mathbb R}
\title{Deformations of the discrete Heisenberg group}
\author{Severin {\sc Barmeier}}
\address{
Instituto de Matem\'atica, Estat\'istica e Computa\c{c}\~ao Cient\'ifica, Universidade Estadual de Campinas\\Rua S\'ergio Buarque de Holanda, 651, Cidade Universit\'aria ``Zeferino Vaz'', Campinas, SP, Brazil}
\begin{document}
\maketitle

\begin{abstract}
We study deformations of the discrete Heisenberg group acting properly discontinuously on the Heisenberg group from the left and right and obtain a complete description of the deformation space.
\end{abstract}

\section{Introduction and statement of main result.}
We will be interested in deformations of the discrete Heisenberg group as a group acting properly discontinuously and cocompactly on a space $X$. The following defines our notion of {\em deformation}.
\begin{defn}[\cite{kobayashi93,kobayashi01,kobayashinasrin06}]
Let $G$ be a Lie group acting continuously on a locally compact space $X$ and let $\Gamma\subset G$ be a discrete subgroup. Define the {\em parameter space} of deformations of $\Gamma$ within $G$, acting properly discontinuously on the space $X$ as
\[
R(\Gamma,G;X)=\left\{\phi\colon\Gamma\to G\middle|
\begin{array}{l}
\phi\text{ is injective},\\
\phi(\Gamma)\text{ acts properly}\\\text{discontinuously}\\
\text{and freely on $X$}
\end{array}\right\}
\]
and the {\em deformation space} as
\[
\mathcal T(\Gamma,G;X)=R(\Gamma,G;X)/G,
\]
where $G$ acts on $R(\Gamma,G;X)$ by conjugation, so that $\mathcal T(\Gamma,G;X)$ is the space of non-trivial deformations.
\end{defn}

There is a natural topology on the parameter space $R(\Gamma,G;X)$ as a subset of $\Hom(\Gamma,G)$ endowed with the compact open topology. Then we consider the quotient topology on the deformation space $\mathcal T(\Gamma,G;X)$ (\cite{kobayashi93,kobayashi01}).

If $X$ is an irreducible Riemannian symmetric space $G/K$, Selberg--Weil rigidity (\cite{weil}) states that $\mathcal T=\mathcal T(\Gamma,G;G/K)$ is discrete if and only if $G$ is not locally isomorphic to $\SL_2\R$. An example of the failure of rigidity is when $G=\PSL_2\R$, $\Gamma$ is the fundamental group of a Riemann surface of genus $g\geqslant 2$ and $X=\SL_2\R/\SO_2$ is the Poincar\'e disk. Then $\mathcal T$ is the Teichm\"uller space, which has dimension $6g-6$.

The study of deformations of discontinuous groups for non-Riemannian homogeneous spaces and the failure of rigidity was initiated by Kobayashi \cite{kobayashi93}; Kobayashi \cite{kobayashi98} treats the case when $G$ is semi-simple. A complete description of the parameter and deformation spaces was first given for $\Gamma=\Z^{k}$ acting on $X=\R^{k+1}$ via some nilpotent group of transformations $G$ in \cite{kobayashinasrin06} and these results were extended to the case where $G$ is the Heisenberg group, $H$ is any connected Lie subgroup and $\Gamma$ is a subgroup acting properly discontinuously and freely on $X=G/H$, in \cite{bakloutikedimyoshino08}.

In this paper, we give a concrete description of the space $R(\Gamma,G\times G;G)$, where $G$ is the Heisenberg group, $\Gamma=G\cap\GL_3\Z$ is the discrete Heisenberg group and the direct product group $G\times G$ acts on the group manifold $G$ from the left and right. Our main result is the following.

\begin{thm}
For the deformation space $\mathcal T(\Gamma,G\times G;G)$ of the discrete Heisenberg group acting properly discontinuously on the group manifold $G$ from the left and right, we have the homeomorphism
\[
\mathcal T(\Gamma,G\times G;G)\cong\GL_2\R\times\R^\times\times\R^3.
\]
\end{thm}

\section{Notation.}\label{notation}
Let $G$ denote the Heisenberg group and $\Gamma=G\cap\GL_3\Z$ denote the discrete Heisenberg group. We will replace the matrix notation by defining
\[
\begin{bmatrix}
a\\b\\c
\end{bmatrix}:=
\begin{pmatrix}
1&a&c\\0&1&b\\0&0&1
\end{pmatrix}.
\]

We will fix a presentation $\Gamma=\langle\gamma_1,\gamma_2\rangle$, where
\begin{align}\label{generators}
\gamma_1=
\begin{bmatrix}
1\\0\\0
\end{bmatrix}
\quad\text{and}\quad
\gamma_2=
\begin{bmatrix}
0\\1\\0
\end{bmatrix}.
\end{align}

As a subgroup $\Gamma$ always acts properly discontinuously and freely on $G$ from the left and the quotient space $\Gamma\backslash G$ is a manifold. Similarly $\Gamma$ always acts properly discontinuously from the right with compact quotient $G/\Gamma$.

To let $\Gamma$ act both from the left and from the right, we rewrite $G$ as the homogeneous space $G\times G/\Delta G$, where $\Delta\colon G\to G\times G$ is the diagonal embedding. Then $\Gamma$ acts on $G\times G/\Delta G$ via homomorphisms $\Gamma\to G\times G$. We note here that $\Hom(\Gamma,G\times G)\cong(G\times G)\times(G\times G)$ as sets, because each generator $\gamma_1,\gamma_2$ can be assigned any element in $G\times G$, as any relations $\gamma_1$ and $\gamma_2$ satisfy as elements of $G$ are also satisfied by any two arbitrary elements in $G\times G$. Via the topology on $G$, then, $\Hom(\Gamma,G\times G)$ can be regarded a topological space. In particular, for $G$ being the Heisenberg group we have that $G\cong\R^3$, whence $\Hom(\Gamma,G\times G)\cong\R^{12}$.

Any homomorphism $\Gamma\to G\times G$ can be written as a pair of homomorphisms $\rho,\rho'\colon\Gamma\to G$. Now write $\Gamma_{\rho,\rho'}=\{(\rho(\gamma),\rho'(\gamma))\mid \gamma\in\Gamma\}$ for the image of the pair $(\rho,\rho')\colon\Gamma\to G\times G$. Then $\Gamma$ acts on $G\times G/\Delta G$ via $\Gamma_{\rho,\rho'}$ and the action of $\Gamma$ on $G$ as subgroup (on the left) is recovered as the action of $\Gamma_{\mathrm{id},\mathbf 1}$ on $G\times G/\Delta G$, where $\mathrm{id}$ is the inclusion and $\mathbf 1$ is the trivial homomorphism. However, for general $\rho,\rho'$ this action is not necessarily properly discontinuous.

\begin{rem}
Rewriting $G$ as $G\times G/\Delta G$ for $G=\widetilde{\SL_2\R}$ allowed Goldman \cite{goldman} to construct non-standard Lorentz space forms. Goldman's conjecture concerning the existence of an open neighbourhood of the embedding ${\rm id}\times{\mathbf 1}$, throughout which the group action remains properly discontinuous was resolved affirmatively for reductive Lie groups by Kobayashi \cite{kobayashi98}. An analogous result holds if $G$ is a simply connected Lie group and $\Gamma$ is a cocompact discrete group by an unpublished result of T.~Yoshino. Our results below show this feature explicitly for $G$ being the Heisenberg group.
\end{rem}

\section{Property (CI) and proper actions.}
To check for proper discontinuity of the action of $\Gamma_{\rho,\rho'}$, we will use a criterion by Nasrin \cite{nasrin01} for 2-step nilpotent groups, which relates properness to the {\em property} (CI).
\begin{defn}[\cite{kobayashi92a}, Def.~6]
We say the triplet $(L,H,G)$ has the property {\rm (CI)} if $L\cap gHg^{-1}$ is compact for any $g\in G$.
\end{defn}
(See \cite{lipsman} for the relationship between the property (CI) and proper actions in the more general context of locally compact topological groups acting on locally compact topological spaces.)
\begin{thm}[\cite{nasrin01}, Thm.~2.11]\label{thmnasrin}
Let $G$ be a simply connected 2-step nilpotent Lie group, and let $H$ and $L$ be connected subgroups. Then the following conditions are equivalent.
\begin{enumerate}
\item $L$ acts properly on $G/H$,
\item the triplet $(L,H,G)$ has the property {\rm (CI)},
\item $L\cap gHg^{-1}=\{e\}$ for any $g\in G$.
\end{enumerate}
\end{thm}
We will apply this theorem to the triple $(L_{\rho,\rho'},\Delta G,G\times G)$, where $G$ is again the Heisenberg group and $L_{\rho,\rho'}$ is the {\em extension} of $\Gamma_{\rho,\rho'}$ defined as follows.
\begin{defn}\label{extension}
Let $\Gamma$ be a discrete subgroup in a Lie group $G$. A connected subgroup $L\subset G$ is said to be the {\em extension} of $\Gamma$ if $L$ contains $\Gamma$ cocompactly.
\end{defn}
The following lemma will allow us to use Thm.~\ref{thmnasrin} to determine the conditions under which $\Gamma_{\rho,\rho'}$ acts properly discontinuously.
\begin{lem}[\cite{kobayashi89}]\label{properproperlydiscontinuous}
Let $L$ be a Lie group acting continuously on a locally compact space $X$. Let $\Gamma\subset L$ be a discrete subgroup such that $\Gamma\backslash L$ is compact. Then the following conditions are equivalent.
\begin{enumerate}
\item $\Gamma$ acts properly discontinuously on $X$,
\item $L$ acts properly on $X$.
\end{enumerate}
\end{lem}
\section{Main results.}
To find the extension of $\Gamma_{\rho,\rho'}$, we use the (global) diffeomorphism $\exp\colon\mathfrak g\to G$, whose inverse we denote by $\log$. Let $\rho,\rho'\colon\Gamma\to G$ be any two homomorphisms. Then $\rho$ and $\rho'$ are determined by their values on the generators, which (in the notation of \S\ref{notation}) we will set to be
\begin{align}\label{rho}
\rho(\gamma_i)=
\begin{bmatrix}
a_i\\b_i\\c_i
\end{bmatrix}
\quad\text{and}\quad
\rho'(\gamma_i)=
\begin{bmatrix}
a'_i\\b'_i\\c'_i
\end{bmatrix},
\end{align}
for $i=1,2$. Now, let $\rho_0\colon\mathfrak g\to\mathfrak g$ be a Lie algebra homomorphism defined on the generators by $\rho_0(\log\gamma_i)=\log\rho(\gamma_i)$, for $i=1,2$, and $\rho_0([\log\gamma_1,\log\gamma_2])=\log\rho([\gamma_1,\gamma_2])$, and extended linearly; let $\overline\rho\colon G\to G$ be defined by $\overline\rho=\exp\circ\rho_0\circ\log$. Then $\overline\rho\vert_\Gamma=\rho$, so that $\overline\rho$ extends $\rho$ in the sense that $\overline\rho$ is defined on all of $G$. If we write $\overline\rho'$ for the extension of $\rho'$ to all of $G$, then $L_{\rho,\rho'}=\{(\overline\rho(g),\overline\rho'(g))\mid g\in G\}$ is the extension of $\Gamma_{\rho,\rho'}$ in the sense of Def.~\ref{extension}.

Next, we will check condition (c) of Thm.~\ref{thmnasrin} for $(L_{\rho,\rho'},\Delta G,G\times G)$. We have that
\begin{align}
&L_{\rho,\rho'}\cap(g_1,g_2)\Delta G(g_1,g_2)^{-1}=\{e\}\notag\\
\Leftrightarrow\;&\overline\rho(g)=g_1^{-1}g_2\overline\rho'(g)(g_1^{-1}g_2)^{-1}\text{ only if $g=e$}\notag\\
\Leftrightarrow\;&\rho_0(\log g)=\mathrm{Ad}_{g_1^{-1}g_2}\rho'_0(\log g)\text{ only if $\log g=0$}\label{condition}
\end{align}
for all $(g_1,g_2)\in G\times G$. Now write 
\[
g=\begin{pmatrix}1&a&c\\0&1&b\\0&0&1\end{pmatrix}
\quad\text{and}\quad\log g=\begin{pmatrix}0&a&c-\frac12ab\\0&0&b\\0&0&0\end{pmatrix}.
\]

Calculating the LHS and RHS of (\ref{condition}) explicitly, it follows that (\ref{condition}) is equivalent to
\begin{align*}
&\begin{pmatrix}
a_1 & a_2 & 0\\
b_1 & b_2 & 0\\
\ast & \ast & a_1b_2-a_2b_1
\end{pmatrix}
\begin{pmatrix}
a\\b\\c
\end{pmatrix}\\=
&\begin{pmatrix}
a'_1 & a'_2 & 0\\
b'_1 & b'_2 & 0\\
\ast & \ast & a'_1b'_2-a'_2b'_1
\end{pmatrix}
\begin{pmatrix}
a\\b\\c
\end{pmatrix}\Rightarrow a=b=c=0.
\end{align*}
Writing
\begin{align}\label{a}
A=\begin{pmatrix}
a_1 & a_2\\b_1 & b_2
\end{pmatrix}
\quad\text{and}\quad
A'=\begin{pmatrix}
a'_1 & a'_2\\b'_1 & b'_2
\end{pmatrix}
\end{align}
we can rewrite condition (\ref{condition}) as
\[
\det\begin{pmatrix}
A-A' & 0\\ \ast & \det A-\det A'
\end{pmatrix}\neq 0,
\]
and we obtain the following proposition.
\begin{prop}\label{mainresult}
The group $\Gamma_{\rho,\rho'}$ acts properly discontinuously and cocompactly on $G\times G/\Delta G$ if and only if the following two conditions hold.
\begin{enumerate}
\item $\det(A-A')\neq0$, and
\item $\det A-\det A' \neq0$,
\end{enumerate}
where $A,A'$ are determined by $\rho,\rho'$ via (\ref{rho}) and (\ref{a}).
\end{prop}
Proper discontinuity is contained in the above argument. For cocompactness we make use of the following lemma.
\begin{lem}\label{rhoinjective}
Let $\rho$ be as in (\ref{rho}) and $A$ be defined by (\ref{a}). Then $\det A\neq0\Leftrightarrow\rho$ is injective.
\end{lem}
\begin{proof}
$\det A$ is precisely the (1,3) entry of the commutator $[\rho(\gamma_1),\rho(\gamma_2)]$ and $\det A\neq0$ if and only if the image $\rho(\Gamma)$ is non-commutative. We show that $\rho(\Gamma)$ being non-commutative is equivalent to $\rho$ being injective.

If $\rho$ is injective, $\rho(\Gamma)\cong\Gamma$ is non-commutative. Conversely, write $N=\ker\rho$ and assume that $\rho(\Gamma)$ is non-commutative. We have the commutative diagram
\begin{align*}
\begin{tikzpicture}[baseline=-2.3pt,description/.style={fill=white,inner sep=2pt}]
\matrix (m) [matrix of math nodes, row sep=2.85em,
column sep=1em, inner sep=4pt, text height=1.5ex, text depth=0.25ex, ampersand replacement=\&]
{
	\& 0	\&	0	\& 0	\& 	\\
0	\& N\cap\Z	\&	N	\& N/N\cap\Z	\& 0	\\
0	\& \Z	\&	\Gamma	\& \Z^2	\& 0	\\
0	\& \Z/N\cap\Z	\&	\Gamma/N	\& \Gamma/\Z N	\& 0	\\
	\& 0	\&	0	\& 0	\& 	\\
};
\draw[-stealth,font=\scriptsize]
(m-1-2) edge node[auto] {} (m-2-2)
(m-1-3) edge node[auto] {} (m-2-3)
(m-1-4) edge node[auto] {} (m-2-4)
(m-2-1) edge node[auto] {} (m-2-2)
(m-2-2) edge node[auto] {} (m-2-3)
(m-2-3) edge node[auto] {} (m-2-4)
(m-2-4) edge node[auto] {} (m-2-5)
(m-2-2) edge node[auto] {} (m-3-2)
(m-2-3) edge node[auto] {} (m-3-3)
(m-2-4) edge node[auto] {} (m-3-4)
(m-3-1) edge node[auto] {} (m-3-2)
(m-3-2) edge node[auto] {} (m-3-3)
(m-3-3) edge node[auto] {} (m-3-4)
(m-3-4) edge node[auto] {} (m-3-5)
(m-3-2) edge node[auto] {} (m-4-2)
(m-3-3) edge node[auto] {} (m-4-3)
(m-3-4) edge node[auto] {} (m-4-4)
(m-4-1) edge node[auto] {} (m-4-2)
(m-4-2) edge node[auto] {} (m-4-3)
(m-4-3) edge node[auto] {} (m-4-4)
(m-4-4) edge node[auto] {} (m-4-5)
(m-4-2) edge node[auto] {} (m-5-2)
(m-4-3) edge node[auto] {} (m-5-3)
(m-4-4) edge node[auto] {} (m-5-4);
\end{tikzpicture}
\end{align*}
whose rows and columns are exact by the nine lemma. Turning our attention to the first column, the top left entry $N\cap\Z$ can be considered as a subgroup of $\Z$, and is thus equal to (\emph{i}) $0$, (\emph{ii}) $\Z$, or (\emph{iii}) $m\Z$, for some $m\geq2$.

{\bf Case {\rm (\emph{ii})}.}\;
If $N\cap\Z=\Z$, $N$ contains the commutator $\Z=[\Gamma,\Gamma]$, contradicting the fact that $\rho(\Gamma)\cong\Gamma/N$ was assumed non-commutative.

{\bf Case {\rm (\emph{iii})}.}\;
If $N\cap\Z=m\Z$, for $m\geq2$, then $\Z/N\cap\Z=\Z_m$ in the bottom left entry. However, $\Z_m$ is finite and contains torsion elements and injects into $\Gamma/N$. By the first isomorphism theorem for groups, the induced map $\rho_\ast\colon\Gamma/N\to G$ is injective. But $G$ is torsion-free, whence $\Gamma/N$ is torsion-free also and we obtain a contradiction.

We conclude that $N\cap\Z=0$ (case {\rm (\emph{i})}).

Now, write $\pi\colon\Gamma\to\Z^2$ for the projection and $\pi^\ast\colon N\to N/N\cap\Z$ for the restriction of $\pi$ to $N$. Let $\gamma$ be any element in $\Gamma$ and $n\in N$. Since $N$ is normal, $\gamma n\gamma^{-1}\in N$. Then 
\[
\pi^\ast(\gamma n\gamma^{-1})=\pi^\ast(\gamma)\pi^\ast(n)\pi^\ast(\gamma^{-1})=\pi^\ast(n),
\]
where the last equality follows from the fact that $\mathrm{im}\,\pi^\ast$ injects into $\Z^2$ and is therefore commutative. Since $N\cap\Z=0$, $\pi^\ast$ is an isomorphism and we conclude that $\gamma n=n\gamma$, i.e.~$N$ is contained in the centraliser $\Z$. Then $N\cap\Z=0$ shows that $N$ is trivial, whence $\rho$ is injective.
\end{proof}
\begin{proof}[Proof of Prop.~\ref{mainresult}.]
Using Thm.~\ref{thmnasrin}, we have shown that $L_{\rho,\rho'}$ acts properly on $G\times G/\Delta G$ if and only if conditions (a) and (b) hold. Applying Lem.~\ref{properproperlydiscontinuous}, $L_{\rho,\rho'}$ acts properly on $G\times G/\Delta G$ if and only if $\Gamma_{\rho,\rho'}$ acts properly discontinuously on $G\times G/\Delta G$.

By Lem.~\ref{rhoinjective}, condition (b) shows that at least one of $\rho,\rho'$ must be injective, whence the cohomological dimension $\cd\Gamma_{\rho,\rho'}=3$. It is a fact, based on a standard argument invoking Poincar\'e duality, that if a group $\Gamma$ acts (faithfully) on a contractible manifold $X$ and $\cd\Gamma=\dim X$, then $\Gamma\backslash X$ is compact (cf.~\cite{kobayashi89}, Cor.~5.5). Since $G\times G/\Delta G\cong\R^3$ is indeed contractible and $\dim G\times G/\Delta G=\cd\Gamma_{\rho,\rho'}=3$, the double quotient $\Gamma_{\rho,\rho'}\backslash G\times G/\Delta G$ is compact.
\end{proof}

Prop.~\ref{mainresult} can be turned into a method for determining pairs of homomorphisms for which $\Gamma_{\rho,\rho'}$ acts properly discontinuously and cocompactly on $G$ from the left and right as follows.

Let
\begin{align*}
S=\begin{pmatrix}s_0 & s_1\\s_2 & s_3\end{pmatrix}&\in\GL_2\R\\
\text{and }(S,t_0,t_1,t_2,t_3,c_1,c_2,c_1',c_2')&\in\GL_2\R\times\R^\times\times\R^7.
\end{align*}
Define a map
\begin{align}\label{alpha}
\alpha\colon\GL_2\R\times\R^\times\times\R^7&\to R(\Gamma,G\times G;G)\\
(S,t_0,t_1,t_2,t_3,c_1,c_2,c_1',c_2')&\mapsto\phi,\notag
\end{align}
where $\phi=(\rho,\rho')$ is defined by
\begin{align*}
\rho(\gamma_1)&=
\begin{bmatrix}
\tfrac12(s_0(t_0+t_3)+s_0+s_1t_2)\\
\tfrac12(s_2(t_0+t_3)+s_2+s_3t_2)\\
c_1
\end{bmatrix}\\
\rho(\gamma_2)&=
\begin{bmatrix}
\tfrac12(s_1(t_0-t_3)+s_1+s_0t_1)\\
\tfrac12(s_3(t_0-t_3)+s_3+s_2t_1)\\
c_2
\end{bmatrix}\\
\rho'(\gamma_1)&=
\begin{bmatrix}
\tfrac12(s_0(t_0+t_3)-s_0+s_1t_2)\\
\tfrac12(s_2(t_0+t_3)-s_2+s_3t_2)\\
c'_1
\end{bmatrix}\\
\rho'(\gamma_2)&=
\begin{bmatrix}
\tfrac12(s_1(t_0-t_3)-s_1+s_0t_1)\\
\tfrac12(s_3(t_0-t_3)-s_3+s_2t_1)\\
c'_2
\end{bmatrix}.
\end{align*}
Determining $A,A'$ via (\ref{a}), one checks that $A-A'=S$ and $\det A-\det A'=t_0\cdot\det S\neq0$, as $t_0\in\R^\times$. Thus, conditions (a) and (b) from Prop.~\ref{mainresult} are satisfied and $\Gamma_{\rho,\rho'}$ acts properly discontinuously and cocompactly on $G$ from the left and right. Moreover, we have the following
\begin{thm}
The map $\alpha$ (see (\ref{alpha})) induces a homeomorphism from $\GL_2\R\times\R^\times\times\R^7$ onto the parameter space $R(\Gamma,G\times G;G)$ of deformations of $\Gamma$ acting properly discontinuously on the group manifold $G$ from the left and right. Furthermore, the deformation space $\mathcal T(\Gamma,G\times G;G)$ is homeomorphic to $\GL_2\R\times\R^\times\times \R^3$.
\end{thm}
\begin{proof}
The idea of the proof and the origin of the map $\alpha$ is the following.

The space of pairs of matrices satisfying (a) and (b) of Prop.~\ref{mainresult} can be determined as follows. Suppose $A,A'$ satisfy (a) and (b). Consider the map
\begin{align*}
\omega\colon(A,A')\mapsto(U,V)=(A-A',(A-A')^{-1}(A+A')),
\end{align*}
which is well-defined, since $U=A-A'$ is invertible. We can find an inverse mapping
\[
\alpha_0\colon(U,V)\mapsto(\tfrac12(UV+U),\tfrac12(UV-U))
\]
and one checks that $\alpha_0\circ\omega=\mathrm{id}$ and $\omega\circ\alpha_0=\mathrm{id}$.

For $U$ and $V$, condition (a) is equivalent to the condition that $U\in\GL_2\R$; condition (b) translates into the condition
\begin{align*}
\det\tfrac12(UV+U)&\neq\det\tfrac12(UV-U)\\
\Leftrightarrow\hspace{30pt}\det (V+I)&\neq\det(V-I),\label{determinanttrace}\\
\Leftrightarrow\hspace{23.7pt}\det V+\tr V&\neq\det V-\tr V\\
\Leftrightarrow\hspace{59.6pt}\tr V&\neq0,
\end{align*}
where $I$ denotes the $2\times2$ identity matrix. Then, writing $M=\{V\in M_2(\R)\mid\tr V\neq 0\}\cong\R^\times\times\R^3$, the map
\[
\alpha_0\colon\GL_2\R\times M\to\{(A,A')\mid A,A'\text{ satify (a) \& (b)}\}
\]
is a homeomorphism. Writing $\mathrm{id}$ for the identity on $\R^4=\{(c_1,c_2,c_1',c_2')|c_1,c_2,c_1',c_2'\in\R\}$,  $\alpha_0\times\mathrm{id}=\alpha$ is the homeomorphism
\[
\alpha\colon\GL_2\R\times\R^\times\times\R^7\to R(\Gamma,G\times G;G)
\]
up to the identification $M\cong\R^\times\times\R^3$.

The conjugation action of $G\times G$ on $\Gamma_{\rho,\rho'}$ leaves the superdiagonal entries of each factor unchanged and is transitive on the $(1,3)$ entries, so that $\mathcal T(\Gamma,G\times G;G)=R(\Gamma,G\times G;G)/(G\times G)$ is homeomorphic to
\begin{align*}
\GL_2\R\times\R^\times\times\R^3.
\end{align*}
\end{proof}
\section{Geometric interpretation of main result.}
Geometrically speaking, we have the central extensions
\begin{align*}
\begin{tikzpicture}[baseline=-2.3pt,description/.style={fill=white,inner sep=2pt}]
\matrix (m) [matrix of math nodes, row sep=0.2em,
column sep=3ex, inner sep=2pt, text height=1.5ex, text depth=0.25ex, ampersand replacement=\&]
{0 \& \R \& G \& \R^2 \& 0\\
 0 \& \Z \& \Gamma \& \Z^2 \& 0\\};
\path[-stealth,font=\scriptsize]
(m-1-1) edge (m-1-2)
(m-1-2) edge (m-1-3)
(m-1-3) edge (m-1-4)
(m-1-4) edge (m-1-5)
(m-2-1) edge (m-2-2)
(m-2-2) edge (m-2-3)
(m-2-3) edge (m-2-4)
(m-2-4) edge (m-2-5);
\end{tikzpicture}
\end{align*}
and by quotienting $\Gamma\backslash G$ can be viewed as a circle bundle over the torus. The two conditions of Prop.~\ref{mainresult} can then be interpreted as follows. The matrix $A-A'$ determines a Riemannian structure on this torus and $\det A-\det A'$ determines the structure on (i.e.~length of) the circle. In particular, the number of connected components (which equals four) of the deformation space $\mathcal T(\Gamma,G\times G;G)$ corresponds to the number of possible combinations of orientations on the torus and the circle.

\begin{example}\label{notgraph}
Let
\[
\rho(\gamma_1)=
\begin{bmatrix}
2\\
c\\
0
\end{bmatrix},\qquad
\rho(\gamma_2)=
\begin{bmatrix}
1\\
2\\
0
\end{bmatrix}
\]
and
\[
\rho'(\gamma_1)=
\begin{bmatrix}
1\\
c\\
0
\end{bmatrix},\qquad
\rho'(\gamma_2)=
\begin{bmatrix}
0\\
1\\
0
\end{bmatrix}.
\]
Letting $c$ vary from $0$ to $1$, we obtain a family of groups $\Gamma_{\rho,\rho'}$ (which lies in the component of both base space and fibre orientations being positive), which by Prop.~\ref{mainresult} act cocompactly and properly discontinuously on $G\times G/\Delta G$, where the length of the fibre varies from $3$ to $2$ and the structure on the torus remains unchanged and is given by the matrix
$\left(\begin{smallmatrix}
1 & 1\\0 & 1
\end{smallmatrix}\right)$.
\end{example}

Similarly, it is possible to find families of groups, which only change the structure on the base space, leaving the length of the fibre unchanged; or families, for which both the structure on the base space and the length of the fibre are fixed, but the connection form is deformed.

\begin{rem}
General examples, like the one above, stand in contrast to the case when $G$ is semisimple of real rank $1$---e.g.~$G=\SL_2\R$, $\SO(n,1)$, $\mathrm{SU}(n,1)$, $\mathrm{Sp}(n,1)$---for which any properly discontinuous group for $G\times G/\Delta G$ is a graph up to a finite-index subgroup (\cite{kobayashi93}, Thm.~2 and Rmk.~1).
\end{rem}

\section*{Acknowledgements.}
The author would like to thank Prof.~Taro Yoshino for detailed comments on an earlier version of this paper and Prof.~Toshiyuki Kobayashi for his comments, guidance and invaluable advice.


\begin{thebibliography}{50}
\bibitem[BKY\,'08]{bakloutikedimyoshino08} Baklouti, A., K\'edim, I.~\&~Yoshino, T. ``On the deformation space of Clifford--Klein forms of Heisenberg groups'', {\em Int.~Math.~Res.~Notices}, rnn066: 26 pp., 2008.

\bibitem[G\,'85]{goldman} Goldman, W.M. ``Nonstandard Lorentz space forms'', {\em J.~Diff.~Geom.}, {\bf 21}:301--308, 1985.

\bibitem[K\,'89]{kobayashi89} Kobayashi, T. ``Proper action on a homogeneous space of reductive type'', {\em Math.~Ann.}, {\bf 285}:249--263, 1989.

\bibitem[K\,'92]{kobayashi92a} Kobayashi, T. ``Discontinuous groups acting on homogeneous spaces of reductive type'', in T.~Kawazoe, T.~Oshima \& S.~Sano (eds.) {\em Proc.~of Fuji--Kawaguchiko Conf.~on Representation Th.~of Lie Groups and Lie Algebras}, pp.~59--75, River Edge, NJ: World Scientific, 1992.

\bibitem[K\,'93]{kobayashi93} Kobayashi, T. ``On discontinuous groups acting on homogeneous spaces with noncompact isotropy subgroups'', {\em J.~of Geom.~and Phys.}, {\bf 12}:133--144, 1993.

\bibitem[K\,'98]{kobayashi98} Kobayashi, T. ``Deformation of compact Clifford--Klein forms of indefinite-Riemannian homogeneous manifolds'', {\em Math.~Ann.}, {\bf 310}:395--409, 1998.

\bibitem[K\,'01]{kobayashi01}
Kobayashi, T. ``Discontinuous groups for non-Riemannian homogeneous spaces'', in B.~Engquist \& W.~Schmid (eds.) {\em Mathematics Unlimited---2001 and Beyond}, pp.~723--747, New York, NY: Springer-Verlag, 2001.

\bibitem[KN\,'06]{kobayashinasrin06} Kobayashi, T.~\&~Nasrin, S. ``Deformation of properly discontinuous actions of $\Z^{k}$ on $\R^{k+1}$'', {\em Int.~J.~of Math.}, {\bf 17}(10):1175--1193, 2006.

\bibitem[L\,'95]{lipsman} Lipsman, R.L. ``Proper action and a compactness condition'', {\em J.~of Lie Theory}, {\bf 5}:25--39 1995.

\bibitem[N\,'01]{nasrin01} Nasrin, S. ``Criterion of proper actions for 2-step nilpotent Lie groups'', {\em Tokyo J.~Math.}, {\bf 24}(2):535--543, 2001.

\bibitem[W\,'64]{weil} Weil, A. ``Remarks on the cohomology of groups'', {\em Ann.~Math.}, {\bf 80}:149--157, 1964.\nolinebreak
\end{thebibliography}
\end{document}